\documentclass{elsart3-1}

\usepackage{amssymb, amsmath}
\usepackage[english,francais]{babel}
\usepackage{xcolor}

\newtheorem{e-proposition}[theorem]{Proposition}

\newtheorem{e-definition}[theorem]{Definition\rm}

\newtheorem{theoreme}{Th\'eor\`eme}[section]
\newtheorem{lemme}[theoreme]{Lemme}
\newtheorem{proposition}[theoreme]{Proposition}

\renewcommand{\theequation}{\arabic{equation}}
\newcounter{steps}
\newenvironment{proof}[1][]{%
\par\medbreak\setcounter{steps}{0}
{\noindent\bfseries Preuve#1. }} {\hfill\fbox{\ }\medbreak}

\newcounter{substeps}[steps]

\def\og{\leavevmode\raise.3ex\hbox{$\scriptscriptstyle\langle\!\langle$~}}
\def\fg{\leavevmode\raise.3ex\hbox{~$\!\scriptscriptstyle\,\rangle\!\rangle$}}

\newcommand\R{{\mathbb R}}

\newcommand{\lime}{
\lim _{\varepsilon \searrow 0}}

\newcommand{\calE}{
{\mathcal E}}

\newcommand{\calR}{
{\mathcal R}}

\newcommand{\rotp}{
{\mathcal R}\left ( \frac{\oc t}{\eps} \right ) }

\newcommand{\rotm}{
{\mathcal R}\left (-\frac{\oc t}{\eps} \right ) }

\newcommand{\vorth}{
\;^\bot v}

\newcommand{\tvorth}{
\;^\bot \tilde{v}}

\newcommand{\dd}{
\mathrm{d}}

\newcommand{\dive}{
\mathrm{div}}

\newcommand{\eps}[0]{
\varepsilon}

\newcommand{\oc}[0]{
\omega _c}

\newcommand{\fe}[0]{
f ^\varepsilon}

\newcommand{\Xe}[0]{
X ^\varepsilon}

\newcommand{\Ve}[0]{
V ^\varepsilon}

\newcommand{\Veorth}[0]{
\;^\bot V ^\varepsilon}

\newcommand{\tXe}[0]{
\tilde{X} ^\varepsilon}

\newcommand{\tVe}[0]{
\tilde{V} ^\varepsilon}

\newcommand{\tX}[0]{
\tilde{X}}

\newcommand{\tV}[0]{
\tilde{V}}

\newcommand{\rhoe}[0]{
\rho ^\varepsilon}

\newcommand{\tfe}[0]{
\tilde{f} ^\varepsilon}

\newcommand{\tf}[0]{
\tilde{f}}

\newcommand{\phie}[0]{
\phi ^\varepsilon}

\newcommand{\tphi}[0]{
\tilde{\phi}}

\newcommand{\Be}[0]{
B ^\varepsilon}

\newcommand{\fin}[0]{
f ^{\mathrm{in}}}

\newcommand{\inttxtv}[1]{
\int_{\R^2}\!\! \int _{\R^2} #1 \;\mathrm{d}\tilde{v}\mathrm{d}\tilde{x}}

\newcommand{\inttytw}[1]{
\int_{\R^2}\!\! \int _{\R^2} #1 \;\mathrm{d}\tilde{w}\mathrm{d}\tilde{y}}

\newcommand{\intyw}[1]{
\int_{\R^2}\!\! \int _{\R^2} #1 \;\mathrm{d}w \mathrm{d}y }

\newcommand{\intv}[1]{
\int_{\R^2} #1 \;\mathrm{d}v}

\newcommand{\inty}[1]{
\int_{\R^2} #1 \;\mathrm{d}y}

\newcommand{\ind}[1]{
{\bf 1}_{\{#1\}}}

\newcommand{\tx}[0]{
\tilde{x}}

\newcommand{\tv}[0]{
\tilde{v}}

\newcommand{\ty}[0]{
\tilde{y}}

\newcommand{\tw}[0]{
\tilde{w}}

\newcommand{\avetpi}[0]{
\frac{1}{2\pi}\int _0 ^{2\pi}\!\!\!\!}

\begin{document}

\begin{frontmatter}




%
\selectlanguage{english}
\title{The effective Vlasov-Poisson system for strongly magnetized plasmas}

\vspace{-2.6cm}
\selectlanguage{francais}
\title{Le syst\`eme de Vlasov-Poisson effectif pour les plasmas fortement magn\'etis\'es}



\author[authorlabel1]{Miha\"i BOSTAN}
\ead{mihai.bostan@univ-amu.fr}
\author[authorlabel1]{Aur\'elie FINOT}
\ead{aurelie.finot@univ-amu.fr}
\author[authorlabel1]{Maxime HAURAY}
\ead{maxime.hauray@univ-amu.fr}

\address[authorlabel1]{Institut de Math\'ematiques de Marseille I2M, Centre de Math\'ematiques et Informatique CMI, UMR CNRS 7373, 39 rue Fr\'ed\'eric Joliot Curie, 13453 Marseille  Cedex 13
France}

\begin{abstract}
Nous \'etudions le r\'egime du rayon de Larmor fini pour le syst\`eme de Vlasov-Poisson, dans le cas o\`u la longueur de Debye est \'egale au rayon de Larmor. Le champ magn\'etique est suppos\'e uniforme. Nous restreignons l'\'etude de ce probl\`eme non lin\'eaire au cas bi-dimensionnel. Nous obtenons le mod\`ele limite en appliquant les m\'ethodes de gyro-moyenne cf. \cite{BosAsyAna}, \cite{BosTraEquSin}.  Nous donnons l'expression explicite du champ d'advection effectif de l'\'equation de Vlasov, dans laquelle nous avons substitu\'e le champ \'electrique auto-consistant, via la r\'esolution de l'\'equation de Poisson moyenn\'ee \`a l'\'echelle cyclotronique. Nous mettons en \'evidence la structure hamiltonienne du mod\`ele limite et pr\'esentons ses propri\'et\'es~: conservations de la masse, de l'\'energie cin\'etique, de l'\'energie \'electrique, etc.
\vskip 0.5\baselineskip 
\selectlanguage{francais}
\noindent{\bf Abstract} \vskip 0.5\baselineskip \noindent 
We study the finite Larmor radius regime for the Vlasov-Poisson system. The magnetic field is assumed to be uniform. We investigate this non linear problem in the two dimensional setting. We derive the limit model by appealing to gyro-average methods cf. \cite{BosAsyAna}, \cite{BosTraEquSin}. We indicate the explicit expression of the effective advection field, entering the Vlasov equation, after substituting the self-consistent electric field, obtained by the resolution of the averaged (with respect to the cyclotronic time scale) Poisson equation. We emphasize the Hamiltonian structure of the limit model and present its properties~: conservationss of the mass, kinetic energy, electric energy, etc.
\end{abstract}
\end{frontmatter}
\selectlanguage{english}
\section*{Abridged English version}
\noindent
Motivated by the magnetic confinement fusion, which is one of the main application in plasma physics, we analyse the dynamics of a population of charged particles, under the action of a strong uniform magnetic field. The goal of this note is to study the finite Larmor radius regime, that is, we assume that the particle distribution fluctuates at the Larmor radius scale along the orthogonal directions, with respect to the magnetic field \cite{FreSon98}, \cite{FreSon01}, \cite{Han-Kwan12}. To simplify, we consider the two dimensional setting, {\it i.e.,} $x = (x_1, x_2), v = (v_1, v_2)$, with a magnetic field orthogonal to $x_1Ox_2$. We chose a regime such that 
\begin{enumerate}
\item
The reference time $T$ is much larger than the cyclotronic period (strong magnetic field) {\it i.e.,} 
\begin{equation}
\label{EquLoi1}
T \frac{q|\Be|}{m} \approx \frac{1}{\eps}, \mbox{ with } 0 < \eps << 1~;
\end{equation}
Notice that the above hypothesis writes also $\frac{TV}{\rho _L} \approx \frac{1}{\eps}$, where $V$ is the reference velocity (along the orthogonal directions), and $\rho _L$ is the typical Larmor radius.
\item
The kinetic energy is much larger than the potential energy
\begin{equation}
\label{EquLoi2}
\frac{m |V|^2}{q\phi } \approx \frac{1}{\eps}
\end{equation}
where $m$ is the particle mass, $q$ is the particle charge, and $\phi$ is the reference electric potential.
\item
The Larmor radius is of the same order as the Debye length {\it i.e.,}
\begin{equation}
\label{EquLoi3}
\lambda ^2 _D = \frac{\eps _0 \phi }{n q} \approx \rho ^2 _L.
\end{equation}
Here $\eps _0$ is the electric permittivity of the vacuum and $n$ is the charge concentration. 
\end{enumerate}
Accordingly, the presence density $\fe = \fe (t,x,v)$ and the electric potential $\phie$ satisfy the following Vlasov-Poisson system, up to some multiplicative constants, of order one
\begin{equation}
\label{Equ1}
\partial _t \fe + \frac{1}{\eps}( v \cdot \nabla _x \fe + \oc \vorth \cdot \nabla _v \fe ) - \nabla _x \phie \cdot \nabla _v \fe = 0,\;\;(t,x,v) \in \R_+ \times \R ^2 \times \R^2
\end{equation}
\begin{equation}
\label{Equ2}
- \Delta _x \phie = \rho ^\eps := \intv{\fe (t,x,v)},\;\;(t,x) \in \R _+ \times \R ^2
\end{equation}
\begin{equation}
\label{EquIC}
\fe (0,x,v) = \fin (x,v),\;\;(x,v) \in \R ^2 \times \R^2. 
\end{equation}
Here $\oc$ stands for the rescaled cyclotronic frequency (the real cyclotronic frequency is $\omega _c ^\eps = \oc /\eps$, and the real cyclotronic period is $T_c ^\eps = \frac{2\pi}{\oc ^\eps} = \eps \frac{2\pi}{\oc} = \eps T_c$), and it is assumed constant (uniform magnetic field). For any $v = (v_1, v_2) \in \R^2$, we denote by $\vorth$ the vector $\vorth = (v_2, - v_1) \in \R^2$. We study the stability of the family $(\fe, \phie)_{\eps >0}$, when $\eps$ becomes small. The asymptotic behavior follows by filtering out the fast oscillations of the caracteristic equations for \eqref{Equ1}. It is easily seen that the changes over one cyclotronic period of the quantities $\tx = x + \frac{\vorth}{\oc}$, $\tv = {\mathcal R}(\oc t/\eps)v$, are negligible. We expect that the family $\tfe (t, \tx, \tv) = \fe (t, \tx - {\mathcal R}(-\oc t/\eps) \tvorth /\oc, {\mathcal R}(- \oc t/\eps ) \tv)$ converges, as $\eps$ becomes small, toward some profile $\tf (t, \tx, \tv)$. 
\begin{theoreme}
\label{MainResult1}
Let $\fin = \fin (x,v)$ be a non negative, smooth, compactly supported presence density and $(\fe, \phie)_{\eps >0}$ be the solutions of the Vlasov-Poisson sistem \eqref{Equ1}, \eqref{Equ2}, \eqref{EquIC}. We denote by $\tf = \tf (t, \tx, \tv)$ the solution of
\begin{equation}
\label{EquAveVla}
\partial _t \tf + {\mathcal V}[\tf (t)] (\tx, \tv) \cdot \nabla _{\tx} \tf + {\mathcal A}[\tf (t)] (\tx, \tv) \cdot \nabla _{\tv} \tf = 0,\;\;(t, \tx, \tv) \in \R_+ \times \R ^2 \times \R^2
\end{equation}
with the initial condition
\begin{equation}
\label{EquAveIC}
\tf (0, \tx, \tv) = \fin \left (\tx - \frac{\tvorth}{\oc}, \tv \right ),\;\;(\tx, \tv) \in \R^2 \times \R^2
\end{equation}
where the velocity and acceleration vector fields ${\mathcal V}, {\mathcal A}$ are given by
\[
{\mathcal V}[\tf (t)] (\tx, \tv) = - \frac{^\bot \nabla _{\tx}}{\oc} \tphi [\tf (t)],\;\;{\mathcal A}[\tf (t)] (\tx, \tv) = \oc \;^\bot \nabla _{\tv} \tphi [\tf (t)]
\]
\[
\tphi [\tf(t)] (\tx, \tv) = - \frac{1}{2\pi} \inttytw{\left \{\ln \frac{|\tv - \tw|}{|\oc|} \;\ind{|\tx - \ty| \leq \frac{|\tv - \tw|}{|\oc|} }
+ \ln |\tx - \ty|\;\ind{|\tx - \ty| > \frac{|\tv - \tw|}{|\oc|} }\right \}
\tf (t, \ty, \tw)}.
\]
Therefore $\fe (t,x,v) - \tf(t, x + \vorth/\oc, {\mathcal R} (\oc t/\eps)v) = o(1)$ when $\eps \searrow 0$.
\end{theoreme}


\selectlanguage{francais}
\section{Trajectoires effectives}
\noindent

Ce travail s'inscrit dans le cadre de la mod\'elisation des plasmas de fusion. Nous concentrons notre \'etude au r\'egime du rayon de Larmor fini pour le syst\`eme de Vlasov-Poisson bi-dimensionnel. Sous les hypoth\`eses \eqref{EquLoi1}, \eqref{EquLoi2}, \eqref{EquLoi3}, ce r\'egime est d\'ecrit par \eqref{Equ1}, \eqref{Equ2}, \eqref{EquIC}. La m\'ethode d\'evelopp\'ee ici consiste \`a exprimer le potentiel \'electrique \`a l'aide de la solution fondamentale de l'op\'erateur de Laplace dans $\R^2$, puis ins\'erer cette expression dans les trajectoires de l'\'equation de Vlasov. Nous obtenons alors, \`a l'aide des m\'ethodes classiques de gyro-moyenne \cite{BosAsyAna}, \cite{BosTraEquSin}, \cite{BosGuiCen3D} les trajectoires limites et ainsi, les expressions effectives des champs vitesse et acc\'el\'eration de la nouvelle \'equation de Vlasov, d\'ecrivant le r\'egime asymptotique consid\'er\'e. Pour plus de d\'etails sur les preuves de ces r\'esultats, nous renvoyons \`a \cite{BosFinHau15}.

Notons $e$ la solution fondamentale de l'op\'erateur de Laplace dans $\R^2$
\[
e(z) = - \frac{1}{2\pi} \ln |z|,\;\;z \in \R^2 \setminus \{0\}
\]
c'est-\`a-dire $- \Delta e = \delta _0$ dans ${\mathcal D}^\prime (\R^2)$. Le potentiel \'electrique, solution de l'\'equation de Poisson \eqref{Equ2}, s'\'ecrit donc
\begin{equation}
\label{Equ21}
\phie (t,x) = \inty{e(x-y) \rhoe (t,y) } = \intyw{e(x-y) \fe (t,y,w)}.
\end{equation}
Les \'equations caract\'eristiques de l'\'equation de transport \eqref{Equ1} sont donn\'ees par
\begin{equation*}
\label{Equ22}
\frac{\dd \Xe}{\dd t} = \frac{\Ve (t)}{\eps},\;\;\frac{\dd \Ve }{\dd t} = \oc \frac{\Veorth (t)}{\eps} - \nabla _x \phie (t, \Xe (t)),\;\;(\Xe (0), \Ve(0)) = (x,v).
\end{equation*}
Nous cherchons des quantit\'es qui varient peu sur une p\'eriode cyclotronique. Plus exactement, \`a tout instant fix\'e $t>0$, on introduit le changement de coordonn\'ees
\[
\tx = x + \frac{\vorth}{\oc},\;\;\tv = {\mathcal R} \left (  \frac{\oc t}{\eps}  \right ) v
\]
o\`u $\calR(\theta)$ d\'esigne la rotation de $\R^2$ d'angle $\theta$. On v\'erifie ais\'ement que le d\'eterminant jacobien vaut $1$ et alors ces transformations pr\'eservent la mesure de Lebesgue de $\R^4$ {\it i.e.,} $\dd\tv \dd \tx = \dd v \dd x$. En effet, $\tx$ est le centre du cercle de Larmor d'\'ecrit par une particule passant par $x$ avec la vitesse $v$. Ce centre ne varie pas \`a l'\'echelle du mouvement rapide cyclotronique, correspondant \`a la fr\'equence cyclotronique $\frac{\oc}{\eps}$. Plus exactement on a
\begin{equation}
\label{Equ23}
\frac{\dd \tXe}{\dd t} = - \frac{^\bot \nabla _x \phie }{\oc} (t, \Xe (t)) = - \frac{^\bot \nabla _x \phie}{\oc} \left (t, \tXe (t) - \frac{\calR}{\oc}\left ( -  \frac{\oc t}{\eps} \right ) {^\bot\tVe }(t)\right )
\end{equation}
\begin{equation}
\label{Equ24}
\frac{\dd \tVe}{\dd t} = - \calR \left (   \frac{\oc t}{\eps} \right )
\nabla _x \phie  (t, \Xe (t)) = - \calR \left (   \frac{\oc t}{\eps} \right )\nabla _x \phie  \left (t, \tXe (t) - \frac{\calR}{\oc}\left ( -  \frac{\oc t}{\eps} \right ) {^\bot \tVe} (t)\right ).
\end{equation}
On souhaite remplacer le potentiel \'electrique par l'\'expression de \eqref{Equ21}. On introduit les densit\'es de pr\'esence en les coordonn\'ees $(\tx, \tv)$
\[
\fe (t,x,v) = \tfe (t, \tx, \tv),\;\;\tx = x+ \frac{\vorth}{\oc}, \tv = \rotp{} v .
\]
Ainsi, \eqref{Equ21} conduit \`a 
\begin{align*}
\phie \left (t, \tXe (t) - \frac{\calR}{\oc}\left ( -  \frac{\oc t}{\eps} \right ) {^\bot \tVe }(t)\right ) = \inttytw{\!\!\!\!e \left ( \tXe (t) - \ty - \frac{\calR}{\oc}\left ( -  \frac{\oc t}{\eps} \right ) ^\bot ( \tVe (t) - \tw)  \right )\tfe (t, \ty, \tw)}
\end{align*}
et par cons\'equent, \eqref{Equ23}, \eqref{Equ24} deviennent
\begin{equation}
\label{Equ25}
\frac{\dd \tXe }{\dd t} = - \frac{1}{\oc} \inttytw{{^\bot \nabla} e \left ( \tXe (t) - \ty - \frac{1}{\oc} \rotm{} ^\bot ( \tVe (t) - \tw)  \right )\tfe (t, \ty, \tw)}
\end{equation}
\begin{equation}
\label{Equ26}
\frac{\dd \tVe }{\dd t} = - \rotp{} \inttytw{ \nabla e \left ( \tXe (t) - \ty - \frac{1}{\oc} \rotm{} ^\bot ( \tVe (t) - \tw)  \right )\tfe (t, \ty, \tw)}.
\end{equation}
En prenant la moyenne de \eqref{Equ25} sur la p\'eriode cyclotronique $[t, t+T_c ^\eps]$, avec $T_c ^\eps = \eps \frac{2\pi}{\oc}$, et en introduisant la variable rapide $s = (\tau - t)/\eps$, $\tau \in [t, t + T_c ^\eps]$, nous obtenons
\begin{align}
\label{Equ27}
& \frac{\tXe (t + T_c ^\eps) - \tXe (t)}{T_c ^\eps}  = - \frac{1}{\oc T_c ^\eps}\int _t ^{t + T_c ^\eps} \!\!\!\!\inttytw{\!\!\!\!{^\bot \nabla }e \left( \tXe (\tau) - \ty - \calR \left (- \frac{\oc \tau}{\eps}  \right ) \frac{^\bot ( \tVe (\tau) - \tw)}{\oc} \right ) \tfe(\tau)}\dd \tau \nonumber \\
& = - \frac{1}{\oc T_c} \int _0 ^{T_c} \inttytw{{^\bot \nabla }e \left( \tXe (t + \eps s) - \ty - {\mathcal R}\left (- \frac{\oc t }{\eps} - \oc s \right ) \frac{^\bot ( \tVe (t + \eps s) - \tw)}{\oc} \right ) \tfe(t + \eps s)} \dd s \nonumber \\
& \approx - \frac{^\bot \nabla _{\xi}}{\oc} \inttytw{
{\mathcal E}(\tXe (t) - \ty, \tVe (t) - \tw) \tfe(t, \ty, \tw)}\dd \theta
\end{align}
o\`u la fonction ${\mathcal E}$ est d\'efinie par
\[
{\mathcal E} (\xi, \eta) = \avetpi{} e \left (\xi - \oc ^{-1} \calR (\theta){ ^\bot \eta} \right ) \;\dd \theta,\;\;\xi, \eta \in \R^2.
\]
Nous proc\'edons de la mani\`ere identique pour obtenir, \`a partir de \eqref{Equ26}
\begin{align}
\label{Equ28}
& \frac{\tVe (t + T_c ^\eps) - \tVe (t)}{T_c ^\eps}  = - \frac{1}{ T_c ^\eps}\int _t ^{t + T_c ^\eps} \!\!\!\!\!\!\calR \left ( \frac{\oc \tau}{\eps} \right ) \int{{\nabla }e \left( \tXe (\tau) - \ty - \calR \left (- \frac{\oc \tau}{\eps}  \right ) \frac{^\bot ( \tVe (\tau) - \tw)}{\oc} \right ) \tfe(\tau)}\dd \tw \dd \ty\dd \tau \nonumber \\
& = \!- \frac{1}{T_c} \!\!\int _0 ^{T_c} \!\!\!\!\!\!\calR \left ( \frac{\oc t}{\eps} + \oc s\right )\!\!\!\int{\!\!{\nabla }e \left( \tXe (t + \eps s) - \ty - {\mathcal R}\left (- \frac{\oc t }{\eps} - \oc s \right ) \frac{^\bot ( \tVe (t + \eps s) - \tw)}{\oc} \right ) \tfe(t + \eps s)} \dd \tw \dd \ty\dd s \nonumber \\
& \approx \oc {^\bot \nabla _{\eta}} \inttytw{
{\mathcal E}(\tXe (t) - \ty, \tVe (t) - \tw )\tfe(t, \ty, \tw)}\dd \theta.
\end{align}
En passant \`a la limite dans \eqref{Equ27}, \eqref{Equ28}, quand $\eps \searrow 0$, nous obtenons les trajectoires apr\`es filtration du mouvement rapide cyclotronique
\begin{equation*}
\label{Equ29}
\frac{\dd \tX }{\dd t} = - \frac{ ^\bot \nabla _\xi }{\oc} \inttytw{{\mathcal E}(\tX (t) - \ty, \tV (t) - \tw)\tf(t)},\;\;\frac{\dd \tV }{\dd t} = \oc { ^\bot \nabla _\eta } \inttytw{{\mathcal E}(\tX (t) - \ty, \tV (t) - \tw)\tf(t)}
\end{equation*}
o\`u $\tf = \lime \tfe$ est la distribution limite. Par la suite nous d\'eterminons une expression pour ${\mathcal E}(\xi, \eta)$. Cela r\'esulte de la propri\'et\'e de la moyenne pour les fonctions harmoniques. En effet, si $|\xi | > |\eta |/|\oc|$, la fonction $z \to e(z)$ est harmonique dans l'ouvert $\R ^2 \setminus \{0\}$, contenant le disque ferm\'e, de centre $\xi$ et de rayon $|\eta |/|\oc|$, et par cons\'equent nous avons, gr\^ace \`a la formule de la moyenne
\[
{\mathcal E}(\xi, \eta) = e(\xi) = - \frac{1}{2\pi} \ln |\xi|,\;\;|\xi | > \frac{|\eta|}{|\oc|}.
\]
Plus exactement, on d\'emontre cf. \cite{BosFinHau15}
\begin{lemme}
\label{CalcInt}
Pour tout $\xi, \eta \in \R^2$, nous avons 
\[
\calE (\xi, \eta) = e \left ( \frac{\eta}{\oc}\right ) \ind{|\xi | \leq |\eta|/|\oc|} + e(\xi) \ind{|\xi | > |\eta|/|\oc|}
\]
\[
\nabla _\xi \calE (\xi, \eta)= \nabla e (\xi) \;\ind{|\xi | > |\eta|/|\oc|},\;\;\nabla _\eta \calE (\xi, \eta)= \oc ^{-1} \nabla e \left ( \frac{\eta}{\oc}\right ) \;\ind{|\xi | \leq |\eta|/|\oc|}\;\;\mbox{au sens des distributions}.
\]
\end{lemme}

\section{Le mod\`ele limite}

Nous introduisons le potentiel \'electrique
\begin{align*}
\label{Equ45}
\tphi [\tf (t) ] (\tx, \tv) & = \inttytw{\calE(\tx - \ty, \tv - \tw) \tf(t, \ty, \tw)} \\
& = \inttytw{\left \{ e \left ( \frac{\tv - \tw}{\oc}\right ) \ind{|\tx - \ty | \leq |\tv - \tw|/|\oc|} + e(\tx - \ty) \;\ind{|\tx - \ty | > |\tv - \tw|/|\oc|} \right \}\tf(t, \ty, \tw)} \nonumber 
\end{align*}
et les fonctions
\begin{equation*}
\label{Equ35}
{\mathcal V}[\tf(t)] (\tx, \tv) = - \frac{^\bot \nabla _\xi }{\oc} \inttytw{\calE (\tx - \ty, \tv - \tw) \tf(t,\ty, \tw)} = - \frac{^\bot \nabla _{\tx}}{\oc} \tphi [\tf (t)]
\end{equation*}
\begin{equation*}
\label{Equ36}
{\mathcal A}[\tf(t)] (\tx, \tv) = \oc {^\bot \nabla _\eta } \inttytw{\calE (\tx - \ty, \tv - \tw) \tf(t,\ty, \tw)} = \oc {^\bot \nabla _{\tv}} \tphi [\tf (t)].
\end{equation*}
En d\'erivant sous le signe int\'egral, il est \'egalement possible de repr\'esenter les champs vitesse et acc\'el\'eration sous la forme (cf. Lemme \ref{CalcInt})
\begin{align}
\label{Equ37}
{\mathcal V}[\tf(t)] (\tx, \tv) & = - \frac{1}{\oc} \inttytw{{^\bot \nabla _\xi}\calE (\tx - \ty, \tv - \tw) \tf(t,\ty, \tw)} \\
& = - \frac{1}{\oc} \inttytw{{^\bot \nabla e} (\tx - \ty) \;\ind{|\tx - \ty | > |\tv - \tw|/|\oc|} \tf(t, \ty, \tw)} \nonumber 
\end{align}
\begin{align}
\label{Equ38}
{\mathcal A}[\tf(t)] (\tx, \tv) & = \oc \inttytw{{^\bot \nabla _\eta}\calE (\tx - \ty, \tv - \tw) \tf(t,\ty, \tw)} \\
& = \oc \inttytw{{^\bot \nabla e} 
\left ( \frac{\tv - \tw}{\oc}\right ) \ind{|\tx - \ty | \leq |\tv - \tw|/|\oc|} \tf(t, \ty, \tw)}. \nonumber 
\end{align}
Les trajectoires limites sont d\'etermin\'ees par les champs vitesse et acc\'el\'eration ${\mathcal V}[\tf], {\mathcal A}[\tf]$
\begin{equation*}
\label{Equ46}
\frac{\dd \tX }{\dd t} = {\mathcal V}[\tf (t)] (\tX (t), \tV (t)),\;\;\frac{\dd \tV }{\dd t} = {\mathcal A}[\tf (t)] (\tX (t), \tV (t)).
\end{equation*}
Les densit\'es de pr\'esence \'etant conserv\'ees le long des trajectoires, nous obtenons
\[
\tfe (t, \tXe (t), \tVe(t)) = \fe (t, \Xe (t), \Ve (t)) = f(0, x, v) = f(0, \tx - \oc ^{-1} \tvorth, \tv)
\]
et par cons\'equent la densit\'e limite $\tf$ est solution de \eqref{EquAveVla}, \eqref{EquAveIC}. Notons que les \'equations caract\'eristiques limites forment un syst\`eme hamiltonien, en les variables conjugu\'ees $(\tx _2, \oc ^{-1} \tv_1)$ et $(\oc \tx_1, \tv_2)$
\[
\frac{\dd \tX_2}{\dd t} = \frac{\partial \tphi [\tf(t)]}{\partial (\oc \tx _1)}(\tX (t), \tV(t)),\;\;\frac{\dd (\oc ^{-1} \tV_1)}{\dd t} = \frac{\partial \tphi [\tf(t)]}{\partial \tv _2}(\tX (t), \tV(t))
\]
\[
\frac{\dd (\oc \tX_1)}{\dd t} = - \frac{\partial \tphi [\tf(t)]}{\partial \tx _2}(\tX (t), \tV(t)),\;\;\frac{\dd  \tV_2}{\dd t} = -\frac{\partial \tphi [\tf(t)]}{\partial (\oc ^{-1} \tv _1)}(\tX (t), \tV(t)).
\]

\section{Quelques propri\'et\'es du mod\`ele limite}

Les champs de vitesse et acc\'el\'eration \'etant \`a divergence nulle
\[
\dive _{\tx} {\mathcal V} [\tf (t)] = - \frac{1}{\oc} \dive _{\tx} \;^\bot \nabla _{\tx} \tphi [\tf (t)] = 0,\;\;\dive _{\tv} {\mathcal A} [\tf (t)] = \oc  \dive _{\tv} \;^\bot \nabla _{\tv} \tphi [\tf (t)] = 0
\]
l'\'equation \eqref{EquAveVla} s'\'ecrit aussi sous la forme conservative
\[
\partial _t \tf + \dive _{\tx} \{ \tf {\mathcal V}[\tf(t)]\} + \dive _{\tv} \{ \tf {\mathcal A}[\tf(t)]\} = 0.
\]
En particulier nous obtenons la conservation de la masse. Plus g\'en\'eralement, nous d\'emontrons le r\'esultat suivant. 
\begin{proposition}
\label{Conserv}
Soit $\tf = \tf (t,\tx,\tv)$ la solution du probl\`eme \eqref{EquAveVla}, \eqref{EquAveIC} et $\psi = \psi (\tx, \tv)$ une fonction int\'egrable par rapport \`a $\tf (0,\tx, \tv)\dd\tv \dd \tx = \fin (\tx - \oc ^{-1} {^\bot \tv}, \tv)\dd \tv \dd \tx$.
\begin{enumerate}
\item
Pour tout $t \in \R_+$ nous avons
\begin{align}
\label{Equ41}
& 2 \frac{\dd }{\dd t} \inttxtv{\psi (\tx, \tv) \tf (t, \tx, \tv) }  = \int _{\R^2} \!\!\int _{\R^2}\!\!\int _{\R^2}\!\!\int _{\R^2}\!\!
\tf(t,\ty,\tw)\tf (t, \tx, \tv) \\
& \times 
\left [ 
\frac{1}{\oc} \left (\nabla _{\ty} \psi (\ty, \tw)   - \nabla _{\tx} \psi (\tx, \tv) \right )\cdot {^\bot  \nabla} e (\tx - \ty) \;\ind{|\tx - \ty| > |\tv - \tw|/|\oc|} 
\right. \nonumber \\
& + 
\left.
\left (\nabla _{\tv} \psi (\tx, \tv)   - \nabla _{\tw} \psi (\ty, \tw) \right )\cdot {^\bot  \nabla} e \left ( \frac{\tv - \tw}{\oc}\right ) \ind{|\tx - \ty| < |\tv - \tw|/|\oc|} 
\right ]
\;\dd\tw \dd\ty \dd \tv \dd \tx. \nonumber 
\end{align}
\item
En particulier, pour tout $t \in \R_+$ nous avons
\[
\inttxtv{\{1, \tx, \tv, |\tx |^2, |\tv |^2\}\tf (t, \tx, \tv)} = \inttxtv{\{1, \tx, \tv, |\tx |^2, |\tv |^2\}\fin (\tx - \oc ^{-1} \tvorth, \tv)}.
\]
\end{enumerate}
\end{proposition}
\begin{proof}
\begin{enumerate}
\item
Pour tout $t \in \R_+$, nous obtenons, gr\^ace aux formules de repr\'esentation \eqref{Equ37}, \eqref{Equ38}
\begin{align*}
\frac{\dd }{\dd t}&  \inttxtv{\psi (\tx, \tv) \tf (t, \tx, \tv) }  =
\inttxtv{\psi(\tx, \tv) \partial _t \tf} \\
& = \inttxtv{\left [\nabla _{\tx} \psi \cdot {\mathcal V}[\tf(t)] + \nabla _{\tv}\psi  \cdot {\mathcal A} [\tf(t)]   \right ]\tf(t, \tx, \tv)} \\
& = - \frac{1}{\oc} \int _{\R^2}\!\!\int _{\R^2}\!\!\int _{\R^2}\!\!\int _{\R^2}\!\!\nabla _{\tx} \psi (\tx, \tv) \cdot {^\bot \nabla} e (\tx - \ty) \;\ind{|\tx - \ty| > |\tv - \tw|/|\oc|} \tf(t,\ty,\tw) \tf(t,\tx,\tv) \;\dd \tw \dd\ty \dd \tv \dd \tx \\
& + \;\;\;\;\int _{\R^2}\!\!\int _{\R^2}\!\!\int _{\R^2}\!\!\int _{\R^2}\!\!\nabla _{\tv} \psi (\tx, \tv) \cdot {^\bot \nabla} e \left(\frac{\tv - \tw}{\oc}\right) \ind{|\tx - \ty| < |\tv - \tw|/|\oc|} \tf(t,\ty,\tw) \tf(t,\tx,\tv) \;\dd \tw \dd\ty \dd \tv \dd \tx.
\end{align*}
La formule \eqref{Equ41} r\'esulte en interchangeant $(\tx, \tv)$ contre $(\ty, \tw)$, combin\'e \`a Fubini.
\item
Les conservations r\'esultent facilement, par \eqref{Equ41} appliqu\'ee successivement aux fonctions $1, \tx, \tv, |\tx |^2, |\tv |^2$. 
\end{enumerate}
\end{proof}
Etant donn\'e que l'\'energie cin\'etique est conserv\'ee, et comme on s'attend \`a ce que l'\'energie globale soit conserv\'ee, nous devrions retrouver aussi la conservation de l'\'energie \'electrique. Effectivement nous d\'emontrons 
\begin{proposition}
\label{ConservElec}
Pour tout $t \in \R_+$ nous avons
\[
\frac{\dd }{\dd t} \frac{1}{2} \inttxtv{\tphi [\tf(t)] (\tx, \tv) \tf(t,\tx, \tv)} = 0.
\]
\end{proposition}
\begin{proof}
L'\'energie \'electrique s'\'ecrit sous la forme
\[
\frac{1}{2}\inttxtv{\tphi [\tf(t)] (\tx, \tv) \tf(t,\tx, \tv)} = \frac{1}{2}\int _{\R^2}\!\!\int _{\R^2}\!\!\int _{\R^2}\!\!\int _{\R^2}\!\!\calE (\tx - \ty, \tv - \tw)\tf(t,\tx,\tv) \tf (t, \ty, \tw)\;\dd \tw \dd \ty \dd \tv \dd \tx
\]
et en utilisant la parit\'e de $\calE (\xi, \eta)$ en les variables $\xi$ et $\eta$, nous obtenons facilement, par Fubini, que
\begin{align*}
\frac{\dd }{\dd t} \frac{1}{2} \inttxtv{\tphi [\tf(t)] (\tx, \tv) \tf(t,\tx, \tv)} & = \inttxtv{\tphi [\tf(t)](\tx, \tv)   \partial _t \tf }\\
& = \inttxtv{\left [ \nabla _{\tx} \tphi [\tf (t)] \cdot {\mathcal V} [\tf(t)] + \nabla _{\tv} \tphi [\tf (t)] \cdot {\mathcal A} [\tf(t)] \right ] \tf  } = 0.
\end{align*}

\end{proof}





\end{document}